\documentclass[oneside,reqno]{amsart}
\usepackage{amsmath, amsthm, mathrsfs, amssymb, amsfonts, mathtools}
\usepackage{graphicx}
\usepackage{hyperref}
\hypersetup{colorlinks=true, linkcolor=blue, citecolor=magenta, filecolor=magenta, urlcolor=cyan}

\usepackage{enumitem}
\usepackage[backend=biber, style=trad-alpha, doi=true, isbn=false, hyperref=true, backref=false, giveninits=true, abbreviate=true, sorting=nyt]{biblatex}
\bibliography{lib}
\def\diff{\mathrm{d}}
\newtheorem{theorem}{Theorem}

\newtheorem{conjecture}{Conjecture}
\theoremstyle{definition}
\newtheorem{defn}{Definition}
\newtheorem{remark}{Remark}

\newtheoremstyle{mystyle}
  {3ex}
  {3ex}
  {}
  {}
  {}
  {}
  {\newline}
  {\thmname{{\bf #1}}\thmnumber{ {\bf #2:}}\thmnote{ #3}}
\theoremstyle{mystyle}
\newtheorem{innercustomthm}{Case}
\newenvironment{customcase}[1]
  {\renewcommand\theinnercustomthm{#1}\innercustomthm}
  {\endinnercustomthm}
  
\newcommand{\Nf}{K}
\newcommand{\Nt}{N}
\newcommand{\Nr}{K}

\usepackage{orcidlink}

\usepackage{cases}

\usepackage{rotating}

\usepackage{empheq}

\usepackage{array}
\newcolumntype{x}[1]{>{\centering\arraybackslash\hspace{0pt}}p{#1}}

\usepackage{booktabs}

\numberwithin{equation}{section}

\begin{document}

\title[Trees, trunks, and branches]{Trees, trunks, and branches -- bifurcation structure of time-periodic solutions to $u_{tt}-u_{xx}\pm u^{3}=0$}

\author{Filip Ficek\orcidlink{0000-0001-5885-7064}}
\address{University of Vienna, Faculty of Mathematics, Oskar-Morgenstern-Platz 1, 1090 Vienna, Austria}
\address{University of Vienna, Gravitational Physics, Boltzmanngasse 5, 1090 Vienna, Austria}
\email[]{filip.ficek@univie.ac.at}
\author{Maciej Maliborski\orcidlink{0000-0002-8621-9761}}
\address{University of Vienna, Faculty of Mathematics, Oskar-Morgenstern-Platz 1, 1090 Vienna, Austria}
\address{University of Vienna, Gravitational Physics, Boltzmanngasse 5, 1090 Vienna, Austria}
\email[]{maciej.maliborski@univie.ac.at}
\thanks{We acknowledge the support of the Austrian Science Fund (FWF) through Project \href{http://doi.org/10.55776/P36455}{P 36455} and the START-Project \href{http://doi.org/10.55776/Y963}{Y 963}.
}

\keywords{Time-periodic solutions; Nonlinear wave equation; Bifurcations}
\subjclass{Primary: 35B10; Secondary: 35B32, 35L71}

\begin{abstract}
	We propose a systematic approach to analysing the complex structure of time-periodic solutions to the cubic wave equation on an interval with Dirichlet boundary conditions first reported in \cite{Ficek.2024}. The analysis we present is based on a detailed study of sparse mode interactions suggested by the previous numerical work. Our results complement prior rigorous existence proofs and suggest that solutions exist for any frequency, however, they may be arbitrarily large.
\end{abstract}

\date{\today}

\maketitle

\tableofcontents

\section{Introduction}
\label{sec:Introduction}

The goal of this work is to provide a systematic description of the complex structure of time-periodic solutions first observed in \cite{Ficek.2024}, see also \cite{Arioli.2017} for related studies. In \cite{Ficek.2024} we studied time-periodic solutions to 1d cubic wave equation
\begin{equation}
	\label{eq:23.02.14_01}
	\partial_{t}^{2}u - \partial_{x}^{2}u \pm u^{3} = 0\,,
	\quad (t,x)\in\mathbb{R}_{+}\times [0,\pi]
	\,,
\end{equation}
with Dirichlet boundary conditions
\begin{equation}
	\label{eq:23.02.14_01a}
	u(t,0)=0=u(t,\pi)\,,
	\quad t\in\mathbb{R}_{+}
	\,.
\end{equation}
Following a series of works \cite{Vernov.1998,Khrustalev.2000,Khrustalev.2001}, we carried out a perturbative expansion in amplitude $\varepsilon$ and rigorously obtained solutions up to and including order $\varepsilon^{4}$, at the same time extending the previous works. This result was used to verify a spectral numerical scheme in the regime of small amplitude solutions. Subsequently, using the pseudo-arc-length continuation method \cite{Allgower.1990, Keller.1987}, we studied large solutions. In particular, we observed that solutions bifurcating from the linearised frequency $\Omega=1$ form a complex pattern on the energy-frequency diagram. Here we propose a systematic approach to studying the bifurcation structure of time-periodic solutions to \eqref{eq:23.02.14_01}-\eqref{eq:23.02.14_01a}, which provides a deeper understanding of the results of \cite{Ficek.2024} and gives arguments supporting the claims made there.

Our investigations are intended to complement the classic proofs of existence of small amplitude periodic solutions. Over the years such results have been obtained via various methods such as averaging techniques \cite{bambusi2001families}, Nash-Moser theorem \cite{Berti.2006, berti2007nonlinear}, variational principle \cite{berti2008cantor}, or Lindstedt series techniques \cite{GM.2004, GMP.2005}. Regardless of the method, the significant obstacle that one encounters is the small divisor problem, and overcoming it leads to nowhere dense sets of frequencies of the solutions. The goal of this work is to attempt another approach that leads us to a more complete picture, where the mentioned gaps in frequencies seem to be filled with large amplitude solutions.

Our strategy is to rewrite the PDE \eqref{eq:23.02.14_01} as an infinite algebraic system, the Galerkin system, involving the Fourier coefficients of the solution $u$ and the oscillation frequency $\Omega$, and then to study specific small subsystems of this infinite set of equations. We refer to these smaller subsystems as reducible systems. Their solutions produce a reducible tree, made of trunks and branches, see the definitions in Sec.~\ref{sec:ReducibleSystems}, which effectively replicates the overall pattern along which solution of the Galerkin system follows.

Based on the detailed analysis of the most relevant structures of the reducible tree and comparison with solutions to the truncated Galerkin systems we formulate the
\begin{conjecture}
	The infinite reducible tree and its rescalings according to \eqref{eq:24.05.13_01}, provide an effective description
    for the structure of time-periodic solutions to \eqref{eq:23.02.14_01}-\eqref{eq:23.02.14_01a}.
\end{conjecture}

This conjecture suggests that the time-periodic solutions exist for every frequency $\Omega>0$. They form a complex structure that consists of infinitely many bifurcation points densely populating each solution family originating from linearised frequency. Branches bifurcating from these points connect suitably rescaled solution families and together form a ``fractal-like'' pattern.

Consequently, we claim that there exist solutions with frequencies arbitrarily close to the frequency of the linear problem but which have arbitrarily large energy. Such solutions are also densely distributed in every neighbourhood of linearised frequencies. Since they belong to the non-perturbative regime they are not included in the set of solutions discussed in the cited literature.

This manuscript is structured as follows. In Sec.~\ref{sec:Preliminaries} we pose the problem in terms of the Galerkin approximation. Next, in Sec.~\ref{sec:ReducibleSystems}, we introduce the reducible systems and analyse their lowest dimensional versions -- the essential components of the reducible tree. In Sec.~\ref{sec:ReducibleTree} we analyse the reducible tree and compare it with solutions of the truncated Galerkin systems of increasing size. These observations are essential to back up our conjecture. Perturbations of the reducible systems are analysed in Sec.~\ref{sec:PerturbationsOfReducibleSystems}.
We demonstrate there how the inclusion of additional modes modifies the reducible tree without changing its overall structure. 
The Appendix~\ref{sec:TwoModeSystemAnalysis} completes the analysis of the two-modes systems from Sec.~\ref{sec:ReducibleSystems}.

\section{Preliminaries}
\label{sec:Preliminaries}

\subsection{Statement of the problem}
\label{sec:StatementOfTheProblem}

We are looking for real smooth functions $u=u(\tau,x)$ defined on $(\tau,x)\in[0,2\pi]\times[0,\pi]$ that are solutions to
\begin{equation}
	\label{eqn:wave_eqn}
	\Omega^2 \partial_\tau^{2} u - \partial_x^2 u \pm u^3 = 0
	\,,
\end{equation}
satisfying Dirichlet boundary conditions
\begin{equation}
	\label{eqn:Dirichlet}
	u(\tau,0)=0=u(\tau,\pi)\,,\quad \tau\in[0,2\pi]\,,
\end{equation}
and being $2\pi$-periodic in time
\begin{equation}
	\label{eqn:periodicity}
	u(0,x)=u(2\pi,x)\,, \quad x\in[0,\pi]\,.
\end{equation}
Note that \eqref{eq:23.02.14_01} and \eqref{eqn:wave_eqn} are related by the time coordinate transformation
\begin{equation*}
	\label{eq:24.06.13_01}
	\tau = \Omega t
	\,.
\end{equation*}
For convenience we fix the phase of the solution such that $\partial_{\tau}u|_{\tau=0}=0$.
The conserved energy of solution to \eqref{eqn:wave_eqn} is
\begin{equation}
	\label{eq:24.05.08_01}
	E[u] = \int_{0}^{\pi}\left(\frac{1}{2}\Omega^{2}\left(\partial_{\tau}u\right)^{2}+\frac{1}{2}\left(\partial_{x}u\right)^{2}\pm\frac{1}{4}u^{4}\right)\diff{x}\,.
\end{equation}

The following rescaling
\begin{equation}
	\label{eq:24.05.13_01}
	(u(\tau,x),\Omega)\rightarrow (n\,u(m\tau,nx),n\Omega/m)
	\,,
	\quad
	n,m\in\mathbb{N}
	\,,
\end{equation}
is a symmetry of the PDE \eqref{eqn:wave_eqn}-\eqref{eqn:periodicity}. Therefore, in the analysis, we focus on solutions bifurcating from frequency $\Omega=1$, since solutions bifurcating from higher eigenfrequencies of the linearized problem can be obtained via this scaling.
%
Moreover, the following transformation
\begin{equation*}
	\label{eq:24.05.13_02}
	(u(\tau,x),\Omega) \rightarrow (\Omega^{-1}u(x-\pi/2,\tau-\pi/2),\Omega^{-1})
	\,,
\end{equation*}
relates time-periodic solutions of defocusing and focusing equations. Furthermore, all results of our analysis for defocusing equation can be easily converted to the focusing case using this transformation. Thus, we focus exclusively on the defocusing equation, i.e. Eq.~\eqref{eqn:wave_eqn} with a plus sign.

\subsection{Galerkin approach}
\label{sec:GalerkinApproach}

We write a solution to \eqref{eqn:wave_eqn} as a double Fourier series
\begin{equation}
	\label{eq:24.05.10_01}
	u(\tau,x) = \sum_{k\geq 0}\sum_{j\geq 0} a_{kj}\cos{(2k+1)\tau}\,\sin{(2j+1)x}
	\,,
\end{equation}
with $a_{kj}\in\mathbb{R}$. Plugging \eqref{eq:24.05.10_01} into \eqref{eqn:wave_eqn} and projecting onto $\cos{(2m+1)\tau}\,\sin{(2n+1)x}$ we obtain an infinite algebraic system for the expansion coefficients
\begin{multline}
	\label{eq:24.05.10_02}
	\frac{\pi^{2}}{4}\left(-(2m+1)^{2}\Omega^{2} + (2n+1)^{2}\right) a_{mn}
	\\
	+ \sum_{\substack{m_1,m_2,m_3\\n_1,n_2,n_3}}C_{2m_1+1,2m_2+1,2m_3+1,2m+1}S_{2n_1+1,2n_2+1,2n_3+1,2n+1}a_{m_1,n_1}a_{m_2,n_2}a_{m_3,n_3}=0
	\,,
\end{multline}
where the indices go over non-negative integers and the projection of the cubic term is expressed in terms of the integrals
\begin{multline*}
	C_{ijkl} \equiv \int_{0}^{\pi}\cos{i x}\,\cos{j x}\,\cos{k x}\,\cos{l x}\,\diff{x}
	\\
	= \frac{\pi}{8}\left(\delta_{-i+j+k,l}+\delta_{i-j+k,l}+\delta_{i+j-k,l}+\delta_{i+j+k,l}\right.
	\\
	\left. +\delta_{-i-j+k,l}+\delta_{-i+j-k,l}+\delta_{i-j-k,l}+\delta_{-i-j-k,l}\right)
	\,,
\end{multline*}
and
\begin{multline*}
	S_{ijkl} \equiv \int_{0}^{\pi}\sin{i x}\,\sin{j x}\,\sin{k x}\,\sin{l x}\,\diff{x}
	\\
	= \frac{\pi}{8}\left(\delta_{-i+j+k,l}+\delta_{i-j+k,l}+\delta_{i+j-k,l}-\delta_{i+j+k,l}\right.
	\\
	\left. -\delta_{-i-j+k,l}-\delta_{-i+j-k,l}-\delta_{i-j-k,l}+\delta_{-i-j-k,l}\right)
	\,,
\end{multline*}
which follow from elementary trigonometric identities. We refer to \eqref{eq:24.05.10_02} as the \textit{Galerkin system}. The energy of the solution given by the series \eqref{eq:24.05.10_01} can be expressed in terms of the Fourier coefficients $a_{kj}$. Plugging \eqref{eq:24.05.10_01} into \eqref{eq:24.05.08_01} and evaluating the integrand at $\tau=\pi/2$ we obtain
\begin{equation}
	\label{eq:24.05.11_03}
	E = \frac{\pi}{4}\Omega^{2} \sum_{n\geq 0}\left(\sum_{m\geq 0}a_{mn}(-1)^{m}(2m+1)\right)^{2}
	\,.
\end{equation}

In the following we investigate finite subsystems of the Galerkin system which we simply call \textit{finite Galerkin systems}.
We obtain a finite Galerkin system spanned by modes $A_1 \cos(2m_1+1)\tau \, \sin(2n_1+1) x$, ..., $A_\Nf \cos(2m_\Nf+1)\tau \, \sin(2n_\Nf+1) x$ by plugging the sum
\begin{equation*}\label{eqn:u_finite}
	u(\tau, x)=\sum_{k=1}^\Nf A_k \, \cos(2m_k+1)\tau \, \sin(2n_k+1)x
\end{equation*}
into \eqref{eqn:wave_eqn} and projecting the resulting equation onto $\cos(2m_k+1)\tau\, \sin(2n_k+1)x$, where $k\in\{1,...,\Nf\}$. Energy of solutions of such finite subsystems is given by \eqref{eq:24.05.11_03} with $a_{m_k n_k}=A_k$ and the remaining $a_{mn}$ set to zero.

To provide an effective description of bifurcation structures observed in \cite{Ficek.2024} we consider a special class of finite Galerkin systems introduced in the next section. They are amenable for a complete analytic description, however the analysis becomes increasingly complicated with their size. Our main focus will be on such systems that contain the mode $\cos\tau \,\sin x$. This is because we are interested in solutions bifurcating from $\Omega=1$, and this mode is their essential component.

Presented analysis predicts a very complex structure of time-periodic solutions with infinitely many details. To test our predictions, we compare the results of the following analysis with the numerical solution of the \textit{truncated Galerkin system} which is a subclass of finite Galerkin systems (for details see \cite{Ficek.2024}).
Explicitly, we consider finite Galerkin systems spanned by $\Nt^2$ modes: $a_{00}  \cos \tau\, \sin x$, $a_{01}  \cos \tau\, \sin 3x$, ..., $a_{0,\Nt-1}  \cos \tau\, \sin (2\Nt-1)x$, $a_{10}  \cos 3 \tau\, \sin x$, ..., $a_{\Nt-1,\Nt-1}  \cos (2\Nt-1)\tau\, \sin (2\Nt-1)x$. For definiteness we consider only diagonal truncations, i.e. we take the same number of spatial and temporal modes in the expansion. Solutions to the truncated Galerkin system with truncation $\Nt$ will give us an approximate solutions to \eqref{eqn:wave_eqn} in the form of a truncated double Fourier series
\begin{equation*}\label{eqn:u_finite2}
	u(\tau, x)=\sum_{m=0}^{\Nt-1} \sum_{n=0}^{\Nt-1} a_{mn} \, \cos(2m+1)\tau \, \sin(2n+1)x\, .
\end{equation*}

\section{Reducible systems}
\label{sec:ReducibleSystems}

The key role in our analysis is played by specific finite Galerkin systems that we call reducible. This is backed up by direct comparison with solutions to the truncated Galerkin systems with increasing size and also a rigorous analysis of two-mode Galerkin systems presented in Appendix~\ref{sec:TwoModeSystemAnalysis}. We start with introducing the following definition.

\begin{defn}
Let us consider a finite Galerkin system spanned by $\Nr$ modes: $A_k \cos(2m_k+1)\tau \, \sin(2n_k+1) x$ with $k\in\{1,...,\Nr\}$. We call it a \textit{reducible system} if it has the form
\begin{equation}
	\label{eqn:reducible_system}
	\left\{\begin{aligned}
		A_1\left(9A_1^2+12A_2^2+...+12A_\Nr^2\right)&=16\left[\Omega^2(2m_1+1)^2-(2n_1+1)^2\right]A_1\,,\\
		A_2\left(12A_1^2+9A_2^2+...+12A_\Nr^2\right)&=16\left[\Omega^2(2m_2+1)^2-(2n_2+1)^2\right]A_2\,,\\
		&\mathrel{\makebox[\widthof{=}]{\vdots}}\\
		A_\Nr\left(12A_1^2+12A_2^2+...+9A_\Nr^2\right)&=16\left[\Omega^2(2m_\Nr+1)^2-(2n_\Nr+1)^2\right]A_\Nr\,.
	\end{aligned}\right.
\end{equation}
\end{defn}
Reducible systems are such finite Galerkin systems in which modes mix via cubic nonlinearity only in the most trivial way.
 To illustrate what we mean by this, let us consider an expression $X=\cos k_1+\cos k_2+\cos k_3$, where $k_i$ for $i=1,2,3$ are pairwise distinct. The cubic power of $X$ can then be rewritten as a linear combination of $\cos$ functions with arguments of the form $\pm k_i\pm k_j \pm k_l$, where $i,j,l$ are any combinations of $1,2$, and $3$. It means that regardless of the choice of $k_1$, $k_2$, and $k_3$, the cube $X^3$ will always contain a term $\cos k_1$ coming from combinations like $k_1+k_1-k_1$, as well as $k_1+k_2-k_2$ and $k_1+k_3-k_3$ (and similarly for $\cos k_2$ and $\cos k_3$). Hence, these terms can be called trivial, in contrast to additional ways of getting $\cos k_i$ that are present if the arguments satisfy special relations, e.g., when $k_1=k_2+k_2-k_3$. The same idea can be realised when considering the cubic power of \eqref{eq:24.05.10_01} projected on the $\cos(2m_k+1)\tau \, \sin(2n_k+1) x$ modes. Thus, the trivial terms are always present in addition to terms that require nontrivial conditions on $m_k$ and $n_k$. The reducible systems are these that contain only expressions from the first category. This requirement leads to the \textit{reducibility conditions} giving precise constraints on modes involved in such systems. Those conditions become more stringent, the larger the system is. They will be discussed further in the following subsections.

An important feature of the reducible systems is the fact that fixing any $A_k$ to be zero in \eqref{eqn:reducible_system} leads to the $k$-th equation trivially satisfied, while the remaining equations again build a reducible system (of size $N-1$). This observation motivates us to introduce the following definition.
\begin{defn}
The $n$\textit{-modes solution} to the reducible system \eqref{eqn:reducible_system} is a solution for which exactly $n$ variables $A_k$ are non-zero.
\end{defn}

The $\Nr$-modes solution to Eq.\ \eqref{eqn:reducible_system} is simply described by the linear system of equations:
\begin{equation}\label{eqn:reducible_matrix}
	\frac{3}{16}\begin{pmatrix}
		     3 & 4 & \cdots & \cdots & 4 \\
		    4 &  3 & 4 & \cdots & 4 \\
		\vdots & \ddots & \ddots & \ddots & \vdots \\
		4 & \cdots & 4 & 3 & 4 \\
		4 & \cdots & \cdots & 4 & 3
	\end{pmatrix}
	\begin{pmatrix}
		A_{1}^{2} \\
		A_{2}^{2} \\
		\vdots \\
		A_{\Nr-1}^{2} \\
		A_{\Nr}^{2}
	\end{pmatrix}
	=
	\begin{pmatrix}
		\Omega^{2}(2m_{1}+1)^{2} - (2n_{1}+1)^{2} \\
		\Omega^{2}(2m_{2}+1)^{2} - (2n_{2}+1)^{2} \\
		\hspace{3.4ex}\vdots \\
		\Omega^{2}(2m_{\Nr-1}+1)^{2} - (2n_{\Nr-1}+1)^{2} \\
		\Omega^{2}(2m_{\Nr}+1)^{2} - (2n_{\Nr}+1)^{2} \\
	\end{pmatrix}.
\end{equation}
Since the matrix on the left-hand-side can be easily inverted, the unique solution of this linear system is given by
\begin{equation*}
	\begin{pmatrix}
		A_{1}^{2} \\
		A_{2}^{2} \\
		\vdots \\
		A_{\Nr-1}^{2} \\
		A_{\Nr}^{2}
	\end{pmatrix}
	=
		\begin{pmatrix}
		     \alpha_\Nr & \beta_\Nr & \cdots & \cdots & \beta_\Nr \\
		    \beta_\Nr &  \alpha_\Nr & \beta_\Nr & \cdots & \beta_\Nr \\
		\vdots & \ddots & \ddots & \ddots & \vdots \\
		\beta_\Nr & \cdots & \beta_\Nr & \alpha_\Nr & \beta_\Nr \\
		\beta_\Nr & \cdots & \cdots & \beta_\Nr & \alpha_\Nr
	\end{pmatrix}
	\begin{pmatrix}
		\Omega^{2}(2m_{1}+1)^{2} - (2n_{1}+1)^{2} \\
		\Omega^{2}(2m_{2}+1)^{2} - (2n_{2}+1)^{2} \\
		\hspace{3.4ex}\vdots \\
		\Omega^{2}(2m_{\Nr-1}+1)^{2} - (2n_{\Nr-1}+1)^{2} \\
		\Omega^{2}(2m_{\Nr}+1)^{2} - (2n_{\Nr}+1)^{2} \\
	\end{pmatrix}\, ,
\end{equation*}
where
\begin{equation*}
	\alpha_\Nr=-\frac{16(4\Nr-5)}{3(4\Nr-1)}\, , \quad \beta_\Nr=\frac{64}{3(4\Nr-1)} \, .
\end{equation*}
If we define
\begin{equation*}
	\sigma_{\Nr}=\beta_{\Nr}\sum_{k=1}^\Nr (2m_k+1)^2\, , \quad \xi_{\Nr}=\beta_{\Nr}\sum_{k=1}^\Nr (2n_k+1)^2 \, ,
\end{equation*}
then for any $k\in\{1,...,\Nr\}$ the solution $A_k^2$ can be written explicitly as
\begin{equation}\label{eqn:Ak2_solution}
	A_k^2 = \left[\Omega^2 \left(\sigma_{\Nr}-\frac{16}{3} (2m_k+1)^2 \right)-\left(\xi_{\Nr}-\frac{16}{3} (2n_k+1)^2 \right)\right]\, .
\end{equation}
Since we look for real solutions $A_{k}$, care needs to be taken when computing the square root of \eqref{eqn:Ak2_solution}.
This leads us to the following two remarks.

First, for the $\Nr$-modes solution to exist, the expression inside the square brackets of \eqref{eqn:Ak2_solution} must be positive for every $k\in\{1,...,\Nr\}$. This \textit{positivity condition} imposes some requirements on $\Omega$ that not necessarily can be satisfied for a given choice of modes. It means that in general reducible systems spanned by $\Nr$ modes do not posses $\Nr$-modes solutions.  We discuss this matter in greater detail in the following subsections. At this point let us note, that the positivity condition has a simple geometrical interpretation: since expression on the right-hand-side of \eqref{eqn:Ak2_solution} describes a quadric in $2\Nr$-dimensional space of $(m_1,n_1,...,m_\Nr,n_\Nr)\in\mathbb{R}^{2\Nr}$, with $\Omega$ playing the role of the parameter, the $\Nr$-modes solution exists if and only if there is a value of $\Omega$ such that the intersection of adequate regions determined by these quadrics contains a point with natural coordinates. 

On the second note, a single solution to the linear system \eqref{eqn:reducible_matrix} in $A_k^2$ (assuming the positivity condition holds), obviously leads to $2^\Nr$ $\Nr$-modes solutions to \eqref{eqn:reducible_system} differing just by the sign of each mode amplitude $A_k$.

In the following subsections, we carefully investigate reducible systems of smallest sizes. Special focus will be on systems including the mode $\cos\tau \, \sin x$ (which we call the \textit{fundamental mode}), as they play crucial role in understanding the structure of periodic solutions. In this analysis we will be using the following terminology.

\begin{defn}
The one-mode solutions to the reducible systems, where $\cos(2m+1)\tau\, \sin(2n+1)x$ represents the only non-zero mode, are called \textit{trunks of type} $(m,n)$ or $(m,n)$-\textit{trunks} for short. The trunk of type $(0,0)$ is the \textit{primary trunk}, while the other trunks will be called \textit{secondary trunks}.
\end{defn}

\begin{defn}
When $\Nr\geq 2$, we call a $\Nr$-modes solution to the reducible systems a \textit{branch} and we say that its \textit{order} is $\Nr$. More specifically, if the non-zero modes in this solution are $\cos(2m_1+1)\tau\, \sin(2n_1+1)x,\ldots, \cos(2m_\Nr+1)\tau\, \sin(2n_\Nr+1)x$, we call such a solution a \textit{branch of type} $\{(m_{1},n_{1}),\ldots,(m_{\Nr},n_{\Nr})\}$. Moreover, when a branch of order two contains the fundamental mode $\cos\tau\, \sin x$, we call it a \textit{primary branch}, otherwise it is a \textit{secondary branch}.
\end{defn}

\begin{defn}
Let us fix a truncated Galerkin system with truncation $\Nt$. From all of its $n$-modes reducible subsystems (where $1\leq n \leq \Nt$) we choose all that include the fundamental mode and enjoy the presence of an $n$-modes solution. A collection of branches of such subsystems, together with the primary trunk, constitutes the $\Nt$-\textit{reducible tree}.
\end{defn}

\begin{remark}
As solutions to the finite Galerkin systems will be mainly presented on energy-frequency graphs we will
use the nomenclature introduced above to refer to  the appropriate parts of these graphs.
\end{remark}

\subsection{One-mode systems}

In general, a finite Galerkin system spanned by a single mode $B\cos(2m+1)\tau\, \sin(2n+1)x$ is described by equation
\begin{equation}
	\label{eq:24.05.11_06}
	B\left[9B^{2}-16(2m+1)^{2}\Omega^{2} + 16(2n+1)^{2}\right] = 0
	\,.
\end{equation}
It means that every one-mode system is reducible. Equation \eqref{eq:24.05.11_06} is solved either trivially ($B=0$) or by
\begin{equation}
	\label{eq:24.05.11_07}
	B = \pm \frac{4}{3}\sqrt{(2m+1)^{2}\Omega^{2}-(2n+1)^{2}}
	\,.
\end{equation}
Using the terminology introduced above, this second solution is a trunk of type $(m,n)$. It bifurcates from the zero solution at $\Omega=({2n+1})/({2m+1})$ and its amplitude $B$ increases with frequency $\Omega$ indefinitely. The energy of this solution is given by
\begin{equation*}
	\label{eq:24.05.13_03}
	E = \frac{\pi}{4}\Omega^{2} B^2 (2 m+1)^2
	\,.
\end{equation*}

Since the primary trunk will play a crucial role in the analysis, let us consider it separately. It can be obtained as a one-mode solution to a finite Galerkin system spanned by $A\cos{\tau}\,\sin{x}$ and is given by
\begin{equation}
	\label{eq:24.05.11_05}
	A=\pm \frac{4}{3}\sqrt{\Omega^2-1}
	\,.
\end{equation}
It bifurcates from the zero solution at $\Omega=1$ and its energy is simply $E=\pi\Omega^2 A^2/4$. Note that the symmetry transformation \eqref{eq:24.05.13_01} for the PDE sends a primary trunk to a $(m,n)$-trunk and vice versa.

\subsection{Two-modes systems}
\label{sec:2modes}
We begin the investigation of two-modes reducible systems with those containing the fundamental mode, as they lead to primary branches and so they give first insights into the structure of periodic solutions to the truncated Galerkin system.
Next, for future reference, we discuss two-modes reducible systems including arbitrary modes.

\subsubsection{Reducible system spanned by $A \cos \tau\, \sin x$ and $B \cos (2m+1)\tau\, \sin (2n+1) x$}

In order for a truncated Galerkin system spanned by $A \cos \tau\, \sin x$ and $B \cos (2m+1)\tau\, \sin (2n+1) x$ to be reducible, the second mode must satisfy
\begin{equation}\label{eqn:reducibility_2A}
m\geq 1\, \land\, n\geq 1\, \land \, (m,n)\neq(1,1)\,.
\end{equation}
For completeness, the remaining (nonreducible) cases are investigated in the Appendix~\ref{sec:TwoModeSystemAnalysis}.

Reducible system spanned by $A \cos \tau\, \sin x$ and $B \cos (2m+1)\tau\, \sin (2n+1) x$ has the form
\begin{equation*}
\left\{\begin{aligned}
A \left[9 A^2+12 B^2-16\Omega^2+16\right]&=0\,,
\\
B \left[12 A^2+9 B^2 -16(2m+1)^2\Omega^2 + 16(2n+1)^2\right]&=0
\,,
\end{aligned}\right.
\end{equation*}
and its solutions have energy $E$ given by
\begin{equation*}
E=\frac{\pi}{4}\Omega^2 \left[A^2+(2m+1)^2 B^2 \right]\, .
\end{equation*}
As mentioned before, additionally to the zero solution, primary trunk, and trunk of type $(m,n)$, it may posses also a two-modes solution, i.e., a primary branch. This solution can be written explicitly as
\begin{equation}
\label{eq:24.06.18_01}
\begin{aligned}
A&=\pm \frac{4}{\sqrt{21}} \sqrt{(16m^2+16m+1)\Omega^2-(16n^2+16n+1)}\,,\\
B&=\pm \frac{4}{\sqrt{21}} \sqrt{(12n^2+12n-1)-(12m^2+12m-1)\Omega^2}\,.
\end{aligned}
\end{equation}
Hence, there is a two-modes solution when $\Omega$ satisfies the inequalities
\begin{equation}\label{eqn:2modes_Omega}
     \sqrt{\frac{16n^2+16n+1}{16m^2+16m+1}}< \Omega < \sqrt{\frac{12n^2+12n-1}{12m^2+12m-1}}\,.
\end{equation}
It is easy to check that this condition is satisfied only if $m<n$. If it holds then \linebreak\mbox{$({2n+1})/({2m+1})>1$}, so the bifurcation point of the secondary trunk lays over the primary trunk on the $E-\Omega$ diagram, see Fig.~\ref{fig:2modePlot}. In this situation these trunks are connected by a primary branch. The branch bifurcates from the primary trunk at the upper bound of \eqref{eqn:2modes_Omega} and connects to the secondary trunk at the lower bound. The energy along it decreases with the frequency and stays in the interval
\begin{equation}
	\label{eq:24.06.25_01}
    \frac{(12 n^2+12n-1) }{(12 m^2+12m-1)^2} \leq \left(\frac{16\pi}{3}(n-m) (m+n+1)\right)^{-1}E
    \leq\frac{(2 m+1)^2 (16 n^2+16n+1)}{(16 m^2+16m+1)^2}
    \,.
\end{equation}

\begin{figure}[!t]
	\centering
	\includegraphics[width=0.99\textwidth]{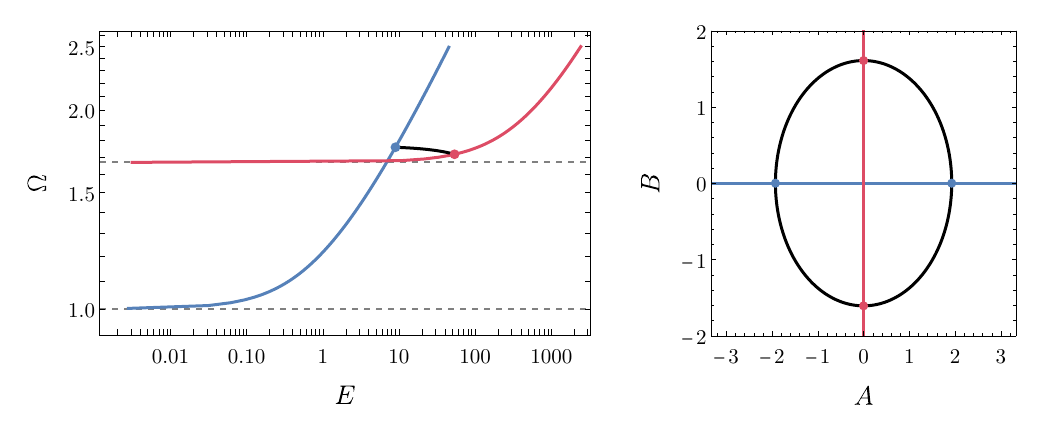}
	\caption{A primary branch (black), spanned by modes $A \cos{\tau}\,\sin{x}$ and $B \cos (2m+1)\tau\, \sin (2n+1) x$, connecting the primary trunk (blue) with a secondary trunk (red). Here we show the $(m,n)=(1,2)$ case, corresponding to the secondary trunk bifurcating from frequency $\Omega=5/3$. On the right plot we show mode amplitudes, $A$ and $B$, along the respective elements of the reducible tree. Note that due to the symmetry of the problem the four solutions (four quarters), see Eq.~\eqref{eq:24.06.18_01}, yield the same curves on the $E-\Omega$ diagram. The blue and red dots indicate bifurcation points on primary ($B=0$) and secondary ($A=0$) trunks, respectively.}
	\label{fig:2modePlot}
\end{figure}

\subsubsection{Reducible system spanned by arbitrary $B_1\cos (2m_1+1) \tau\, \sin (2n_1+1) x$ and \linebreak $B_2 \cos (2m_2+1) \tau\, \sin (2n_2+1) x$}

Even though we focus on reducible systems containing the fundamental mode, in further considerations the understanding of a general two-modes reducible system will prove handy. A careful analysis shows that the system spanned by $B_1\cos (2m_1+1) \tau\, \sin (2n_1+1) x$ and $B_2 \cos (2m_2+1) \tau\, \sin (2n_2+1) x$ is reducible if and only if the following holds
\begin{equation}\label{eqn:reducibility_2B}
m_1\neq m_2 \, \land\, n_1 \neq n_2 \, \land \, (n_2\neq3n_1+1 \lor m_2\neq3m_1+1)\, \land \, (n_1\neq3n_2+1 \lor m_1\neq3m_2+1)\,.
\end{equation}
If these reducibility conditions are satisfied, we get
\begin{subequations}
	\label{eqn:2_modes_general_eqns}
  \begin{empheq}[left=\empheqlbrace]{align}
	B_1 \left[9 B_1^2+12 B_2^2-16(2m_1+1)^2\Omega^2+ 16(2n_1+1)^2\right]&=0\, ,\\
	B_2 \left[12 B_1^2+9 B_2^2-16(2m_2+1)^2\Omega^2+ 16(2n_2+1)^2\right]&=0\, ,
  \end{empheq}
\end{subequations}
with the energy given by
\begin{equation*}
E=\frac{\pi}{4}\Omega^2 \left[(2m_1+1)^2 B_1^2 +(2m_2+1)^2 B_2^2 \right]\, .
\end{equation*}
As we already know, among its solutions there are a zero solution, a trunk of type $(m_1,n_1)$, and a trunk of type $(m_2,n_2)$. Additionally, the potential two-modes solution has the form
\begin{subequations}
	\label{eqn:2_modes_general_sols}
\begin{align}
\label{eqn:2_modes_general_sols1}
B_1&=\pm\frac{4}{\sqrt{21}}\sqrt{\eta(m_2,m_1)\Omega ^2 -\eta(n_2,n_1)}\,,\\
\label{eqn:2_modes_general_sols2}
B_2&=\pm\frac{4}{\sqrt{21}}\sqrt{\eta(m_1,m_2)\Omega ^2 -\eta(n_1,n_2)}\, ,
\end{align}
\end{subequations}
where we have introduced
\begin{align}
\label{eq:24.06.28_01}
\eta(m_1,m_2)=1+16m_1 (m_1+1)-12 m_2(m_2+1)\,.
\end{align}

Next, we study the positivity conditions for this branch. Let us point out that $\eta$ does not have a fixed sign, in particular, one can easily show that
\begin{align*}
\eta(m_1,m_2)>0 \Longleftrightarrow m_1+\frac{1}{2}>\frac{1}{\sqrt{3}}(2m_2+1)\, .
\end{align*}
Additionally, by the definition $\eta$ is an odd number so it can never be zero. It is also impossible to have $\eta(m_1,m_2)<0$ and $\eta(m_2,m_1)<0$ at the same time, since $\eta(m_1,m_2)+\eta(m_2,m_1)=2+4m_1(m_1+1)+4m_2(m_2+1)>0$. It leaves us with three distinct cases that we will now cover.

\begin{customcase}{1}[$\eta(m_2,m_1)>0 \land \eta(m_1,m_2)>0$]
Here we can safely divide both expressions under the square roots in \eqref{eqn:2_modes_general_sols} to get
\begin{equation*}
    \Omega^2> \frac{\eta(n_2,n_1)}{\eta(m_2,m_1)} \qquad \mbox{ and }  \qquad \Omega^2 > \frac{\eta(n_1,n_2)}{\eta(m_1,m_2)}\, .
\end{equation*}
Since it is impossible for both $\eta(n_1,n_2)$ and $\eta(n_2,n_1)$ to be negative, at least one of these bounds is positive. Hence, we have a well defined condition for the existence of the two-modes solution:
\begin{equation*}
    \Omega> \sqrt{\max\left(\frac{\eta(n_2,n_1)}{\eta(m_2,m_1)},\frac{\eta(n_1,n_2)}{\eta(m_1,m_2)}\right)}\, .
\end{equation*}
If the first term in the $\max$ function is larger, this branch bifurcates from the $(m_1,n_1)$-trunk. In the opposite case it comes out from the $(m_2,n_2)$-trunk. The third possibility takes place when both bounds are equal, i.e., $\eta(m_1,m_2)\eta(n_2,n_1)=\eta(m_2,m_1)\eta(n_1,n_2)$, since then both $\eta(n_1,n_2)$ and $\eta(n_2,n_1)$ are positive. This condition is equivalent to $(2n_1+1)/(2m_1+1)=(2n_2+1)/(2m_2+1)$, so both trunks bifurcate from the same frequency. Then the two-modes solution also bifurcates from zero at the same frequency. Independently of the case, after it emerges, this solution exists indefinitely as $\Omega$ increases, with both amplitudes $B_1$, $B_2$ and energy $E$ also increasing.
\end{customcase}

\begin{customcase}{2}[$\eta(m_2,m_1)>0 \land \eta(m_1,m_2)<0$]
In this case the condition for $\Omega$ becomes
\begin{equation}\label{eqn:2_modes_general_Omega}
    \frac{\eta(n_2,n_1)}{\eta(m_2,m_1)}< \Omega^2 < \frac{\eta(n_1,n_2)}{\eta(m_1,m_2)}\, .
\end{equation}
If $\eta(n_1,n_2)>0$ it obviously cannot be satisfied, hence, there is no two-modes solution. In the opposite case, when $\eta(n_1,n_2)<0$, it must hold $\eta(n_2,n_1)>0$. It gives us a possibility of existence of a two-modes solution for some interval of $\Omega$, provided it is nonempty. To check this, let us point out that
\begin{equation*}
    \frac{\eta(n_2,n_1)}{\eta(m_2,m_1)}<\frac{\eta(n_1,n_2)}{\eta(m_1,m_2)}
\end{equation*}
is equivalent to $\eta(m_1,m_2)\eta(n_2,n_1)> \eta(m_2,m_1)\eta(n_1,n_2)$, leading to $(2n_1+1)/(2m_1+1)<(2n_2+1)/(2m_2+1)$. When this condition is satisfied we observe a branch connecting both trunks. This branch emerges from $(m_2,n_2)$-trunk at $\Omega$ resulting from the lower bound of \eqref{eqn:2_modes_general_Omega} and merges with the $(m_1,n_1)$-trunk at the upper bound.
\end{customcase}

\begin{customcase}{3}[$\eta(m_2,m_1)<0 \land \eta(m_1,m_2)>0$]
It is completely analogous to Case 2. This time we observe the two-modes solution if and only if $\eta(n_2,n_1)<0$ and $(2n_1+1)/(2m_1+1)>(2n_2+1)/(2m_2+1)$.
\end{customcase}

In overall, the conditions on different types of behaviour are summarised in Table \ref{tab:2_modes_plots}.

\begin{table}[]
	\setlength\extrarowheight{2.5pt}
    \centering
    \begin{tabular}{|x{0.4cm}x{0.4cm}|x{3.5cm}x{3.5cm}x{3.5cm}|}
    	\toprule
        && $\Omega_1<\Omega_2$ & $\Omega_1=\Omega_2$ & $\Omega_1>\Omega_2$ \\
         \begin{sideways} $\eta(m_2,m_1)>0\qquad$ \end{sideways}& \begin{sideways} $\eta(m_1,m_2)>0$ \end{sideways}& \includegraphics[scale=0.5]{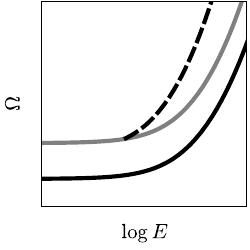} \linebreak branch bifurcating from $(m_2,n_2)$-trunk & \includegraphics[scale=0.5]{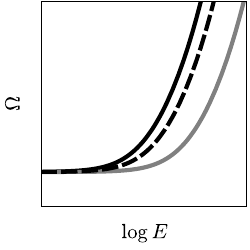} \linebreak branch bifurcating from zero & \includegraphics[scale=0.5]{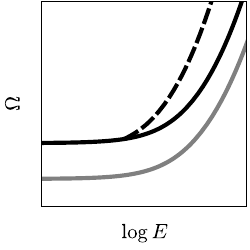} \linebreak branch bifurcating from $(m_1,n_1)$-trunk \\ \midrule

         && $\eta(n_1,n_2)>0$  & $\eta(n_1,n_2)<0$  & $\eta(n_1,n_2)<0$ \\
         & &  $\Omega_1<\Omega_2$ &  $\Omega_1\geq\Omega_2$& \\
         \begin{sideways}   $\eta(m_1,m_2)<0$ \end{sideways}& \begin{sideways}   $\eta(m_2,m_1)>0$ \end{sideways} &\includegraphics[scale=0.5]{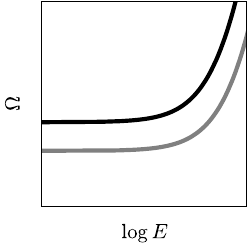} \linebreak no branch & \includegraphics[scale=0.5]{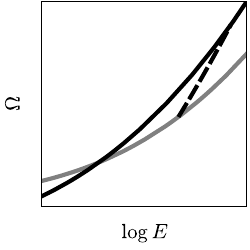} \linebreak branch connecting both trunks & \includegraphics[scale=0.5]{2m_bg.pdf} \linebreak no branch \\ \midrule

         && $\eta(n_2,n_1)>0$  & $\eta(n_2,n_1)<0$  & $\eta(n_2,n_1)<0$ \\
         & &  $\Omega_1>\Omega_2$ &  $\Omega_1\leq\Omega_2$& \\
         \begin{sideways}   $\quad\eta(m_2,m_1)<0$ \end{sideways}& \begin{sideways}   $\quad\eta(m_1,m_2)>0$ \end{sideways} &\includegraphics[scale=0.5]{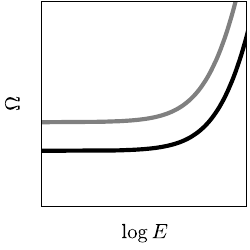} \linebreak no branch & \includegraphics[scale=0.5]{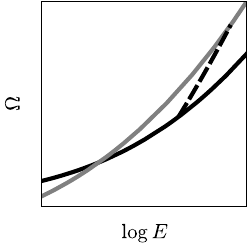} \linebreak branch connecting both trunks & \includegraphics[scale=0.5]{2m_gb.pdf} \linebreak no branch \\ \bottomrule
    \end{tabular}
    \vspace{2ex}
    \caption{Illustrative plots of solutions to \eqref{eqn:2_modes_general_eqns} depending on the trunk types $(m_1,n_1)$ and $(m_2,n_2)$. The solid lines represent one-mode solutions: black is a $(m_1,n_1)$-trunk and gray is a $(m_2,n_2)$-trunk. The dashed lines show the two-modes solutions. To simplify the presentation, we employ here $\Omega_k=(2n_k+1)/(2m_k+1)$ for $k\in\{1,2\}$.}
    \label{tab:2_modes_plots}
\end{table}

\subsection{Three-modes system}

For three-modes reducible systems we focus exclusively on the combinations including the fundamental mode. For a finite Galerkin system spanned by $A \cos \tau \sin x$, $B_1 \cos (2m_1+1) \tau\, \sin (2n_1+1) x$, and $B_2\cos (2m_2+1) \tau\, \sin (2n_2+1) x$ to be reducible, the mode numbers $m_1$, $n_1$, $m_2$, and $n_2$ must satisfy the following set of conditions. First of all, the conditions for reducibility of every two-modes subsystems must hold. It means that we require \eqref{eqn:reducibility_2A} with $(m,n)=(m_k,n_k)$ for $k\in\{1,2\}$, as well as \eqref{eqn:reducibility_2B}. Additionally, the reduction formulas
\begin{align*}
    \cos\tau\, \cos(2m_1+1)\tau \, \cos(2m_2+1)\tau=&\frac{1}{4}\left[\cos(2m_1+2m_2+1)\tau+ \cos(-2m_1+2m_2+1)\tau\right.\\
    & \left.+\cos(2m_1-2m_2+1)\tau +\cos(2m_1+2m_2+3)\tau \right]\,,
    \\
     \sin x\, \sin(2n_1+1)x \, \sin(2n_2+1)x=&\frac{1}{4}\left[\sin(2n_1+2n_2+1)x+ \sin(-2n_1+2n_2+1)x\right. \\
     &\left.+\sin(2n_1-2n_2+1)x - \sin(2n_1+2n_2+3)x \right]\,,
\end{align*}
lead to the following three extra conditions
\begin{multline*}
	(|2m_1-4m_2-1|\neq 1  \lor |2n_1-4n_2-1|\neq 1)  \\
	\land \, (|2m_2-4m_1-1|\neq 1\lor |2n_2-4n_1-1|\neq 1)\\
	\land \, (|m_1-m_2|\neq 1  \lor |n_1-n_2|\neq 1 )\,.
\end{multline*}
If all of them are satisfied, the system is described by
\begin{equation*}
	\left\{\begin{aligned}
	A \left[9 A^2+12 B_1^2+12 B_2^2 - 16\Omega^2 + 16\right]=0\,,\\
	B_1 \left[12 A^2+9 B_1^2+12 B_2^2 - 16(2m_1+1)^2\Omega^2 + 16(2n_1+1)^2\right]=0\,,\\
	B_2 \left[12 A^2+12 B_1^2+ 9 B_2^2 - 16(2m_2+1)^2\Omega^2 + 16(2n_2+1)^2\right]=0\,.\\
	\end{aligned}\right.
\end{equation*}
In addition to a zero solution, three trunks (one primary and two secondary ones), and at most three branches of order two (two primary and one secondary), that fall into the analysis performed above, one can also have a three-modes solution given by
\begin{subequations}
	\label{eqn:reducible_3_sols}
\begin{align}
A&=\pm \frac{4}{\sqrt{33}}\sqrt{\Omega^2\eta_A(m_1,m_2)-\eta_A(n_1,n_2)}\,,\label{eqn:reducible_3_sols1}\\
B_1&=\pm \frac{4}{\sqrt{33}}\sqrt{\Omega^2\eta_B(m_2,m_1)-\eta_B(n_2,n_1)}\,,\label{eqn:reducible_3_sols2}\\
B_2&=\pm \frac{4}{\sqrt{33}}\sqrt{\Omega^2\eta_B(m_1,m_2)-\eta_B(n_1,n_2)}\,,\label{eqn:reducible_3_sols3}
\end{align}
\end{subequations}
where
\begin{align*}
\eta_A(m_1,m_2)&=1+16m_1(m_1+1)+16m_2(m_2+1)\,,\\
\eta_B(m_1,m_2)&=1+16m_1(m_1+1)-28m_2(m_2+1)\,.
\end{align*}

For example, the reducibility and positivity conditions leading to a three-modes solution are satisfied for mode indices $(m_1,n_1)=(1,3)$ and $(m_2,n_2)=(3,8)$, see Fig.~\ref{fig:3modePlot} in which we illustrate the bifurcation structure of the corresponding reducible system.
The goal of the remaining part of this section is to understand for which modes such three-modes solutions exist and what are their characteristics.

\begin{figure}[!t]
	\centering
	\includegraphics[width=0.47\textwidth]{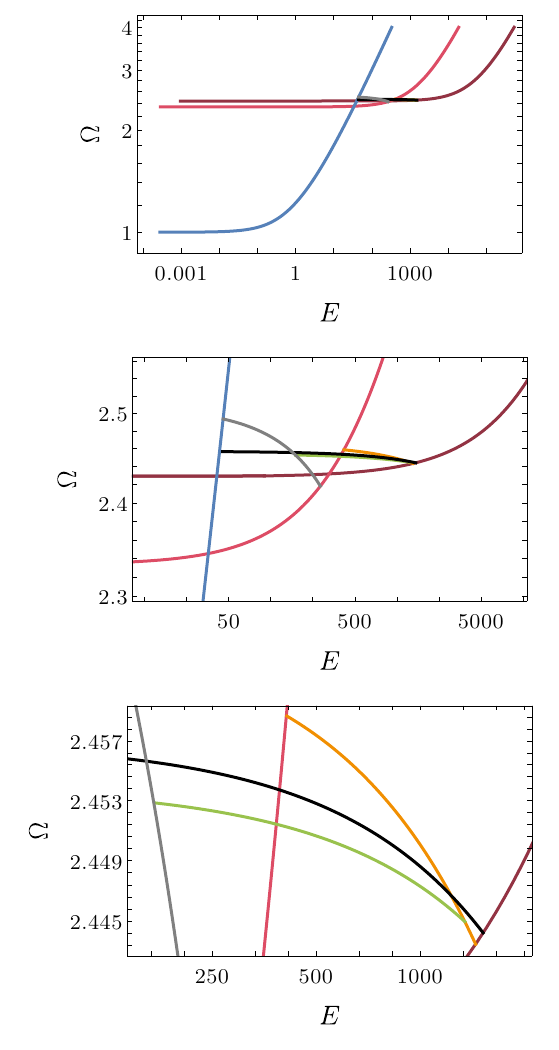}
	\includegraphics[width=0.47\textwidth]{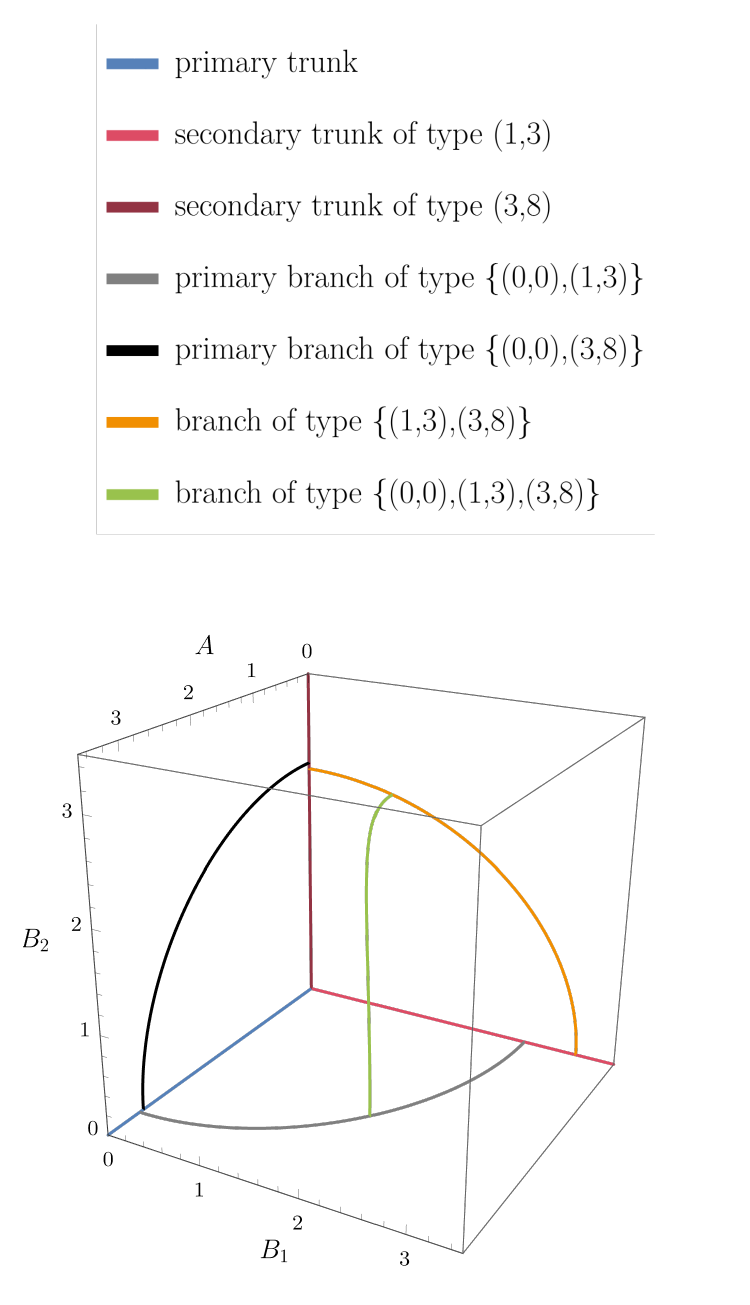}
	\caption{Illustration of higher order structures appearing in the reducible tree: solutions spanned by modes the $A\cos{\tau}\sin{x}$, $B_1\cos (2m_1+1) \tau\, \sin (2n_1+1) x$, and $B_2 \cos (2m_2+1) \tau\, \sin (2n_2+1) x$, with $(m_1,n_1)=(1,3)$ and $(m_2,n_2)=(3,8)$. The primary (blue) and secondary (light and dark red) trunks are joined by the two-modes solutions: the primary branches (black and grey) and branch of type $\{(1,3),(3,8)\}$ (orange). In addition, there exists a three-modes solution (green) connecting one of the primary branches with the branch of order two. The bifurcations points are the end points of the curves corresponding to branches while trunks extend to infinity, cf. \eqref{eq:24.05.11_07} and \eqref{eq:24.05.11_05}. The bifurcation structure is additionally visualised on the 3d diagram of the solution space $(A_{1},B_{1},B_{2})$; for clarity we show only one octant.}
	\label{fig:3modePlot}
\end{figure}

As was the case for $\eta$ from \eqref{eq:24.06.28_01}, the functions $\eta_A$ and $\eta_B$ are odd and so cannot be equal to zero. Additionally, $\eta_A>0$, thus the condition for $A$ in \eqref{eqn:reducible_3_sols1} to be real is simply
\begin{align*}
\Omega^2\geq \frac{\eta_A(n_1,n_2)}{\eta_A(m_1,m_2)}\,.
\end{align*}
It also holds $\eta_B(m_1,m_2)+\eta_B(m_2,m_1)=2-12m_1(m_1+1)-12m_2(m_2+1)$ so, since we require $m_k\geq 1$, this sum is always negative. As a result, it is impossible to have $\eta_B(m_1,m_2)>0$ and $\eta_B(m_2,m_1)>0$ at the same time. The condition of positivity of $\eta_B$ can be reformulated to
\begin{align*}
\eta_B(m_1,m_2)>0 \Longleftrightarrow \sqrt{\frac{2}{7}}\sqrt{2m_1^2+2m_1+1}>m_2+\frac{1}{2}\,.
\end{align*}
In the end, we are again left with three cases to consider.

\begin{customcase}{1}[$\eta_B(m_2,m_1)<0 \land \eta_B(m_1,m_2)<0$]
Here the conditions for all expressions under square roots in \eqref{eqn:reducible_3_sols} to be positive are
\begin{equation*}
    \Omega^2> \frac{\eta_A(n_1,n_2)}{\eta_A(m_1,m_2)}\,, \quad \Omega^2<\frac{\eta_B(n_2,n_1)}{\eta_B(m_2,m_1)}\,, \quad \mbox{ and }  \quad \Omega^2< \frac{\eta_B(n_1,n_2)}{\eta_B(m_1,m_2)}\,.
\end{equation*}
Thus, the potential frequency values must lie inside the interval given by
\begin{equation*}\label{eqn:2_modes_Case1_Omega}
\sqrt{\frac{\eta_A(n_1,n_2)}{\eta_A(m_1,m_2)}} < \Omega < \sqrt{\min\left(\frac{\eta_B(n_2,n_1)}{\eta_B(m_2,m_1)}, \frac{\eta_B(n_1,n_2)}{\eta_B(m_1,m_2)}\right)}\,.
\end{equation*}
This interval is nonempty only if the ratio on the left-hand side (which is always positive) is smaller than both of the ratios on the right-hand side. If there are values of $\Omega$ satisfying this relation, there exists a three-modes solution connecting the secondary branch of type $\{(m_1,n_1),(m_2,n_2)\}$ with one of the primary branches: of type $\{(0,0),(m_2,n_2)\}$ if the first element inside the $\min$ function is smaller, and of type $\{(0,0),(m_1,n_1)\}$ otherwise. In the case when two elements inside the $\min$ function coincide, the three-modes solution emerges from the primary trunk (since $B_1=B_2=0$ at one of its ends).
\end{customcase}

\begin{customcase}{2}[$\eta_B(m_2,m_1)>0 \land \eta_B(m_1,m_2)<0$]
Now the conditions for a three-modes solution to exist are
\begin{equation*}
    \Omega^2> \frac{\eta_A(n_1,n_2)}{\eta_A(m_1,m_2)},  \quad \Omega^2>\frac{\eta_B(n_2,n_1)}{\eta_B(m_2,m_1)} \quad \mbox{ and }  \quad \Omega^2<\frac{\eta_B(n_1,n_2)}{\eta_B(m_1,m_2)}\,.
\end{equation*}
Thus, the possible frequency is inside the interval
\begin{equation*}
\sqrt{\max\left(\frac{\eta_A(n_1,n_2)}{\eta_A(m_1,m_2)},\frac{\eta_B(n_2,n_1)}{\eta_B(m_2,m_1)}\right)} < \Omega <\sqrt{ \frac{\eta_B(n_1,n_2)}{\eta_B(m_1,m_2)}}\,.
\end{equation*}
If it is nonempty, we have a solution connecting the primary branch of type $\{(0,0),(m_1,n_1)\}$ with $\{(m_1,n_1),(m_2,n_2)\}$-branch, if the first term in the $\max$ function is larger, or with primary branch of type $\{(0,0),(m_2,n_2)\}$, in the opposite case. For the situation when the two elements inside the $\max$ function are equal, an analogous discussion as for Case 1 applies.
\end{customcase}

\begin{customcase}{3}[$\eta_B(m_2,m_1)<0 \land \eta_B(m_1,m_2)>0$]
This case is completely analogous to the previous one.
\end{customcase}

\subsection{Larger reducible systems}

These considerations can, in principle, be repeated for four-modes and higher reducible systems including a fundamental mode. By verifying both the reducibility and positivity conditions for all combinations of modes within an increasing range, we find that the lowest branches of order four containing the fundamental mode are of type $\{(0,0),(1,2),(5,9),(20,35)\}$ and $\{(0,0),(1,3),(3,8),(14,35)\}$. These higher-order solutions would be relevant for constructing $\Nt$-reducible trees with $\Nt>35$.
Such solutions introduce even finer structures on the corresponding reducible trees, but their existence does not affect our assertions presented in the next section.

\section{Reducible tree}
\label{sec:ReducibleTree}

\begin{figure}[!t]
	\centering
	\includegraphics[width=0.92\textwidth]{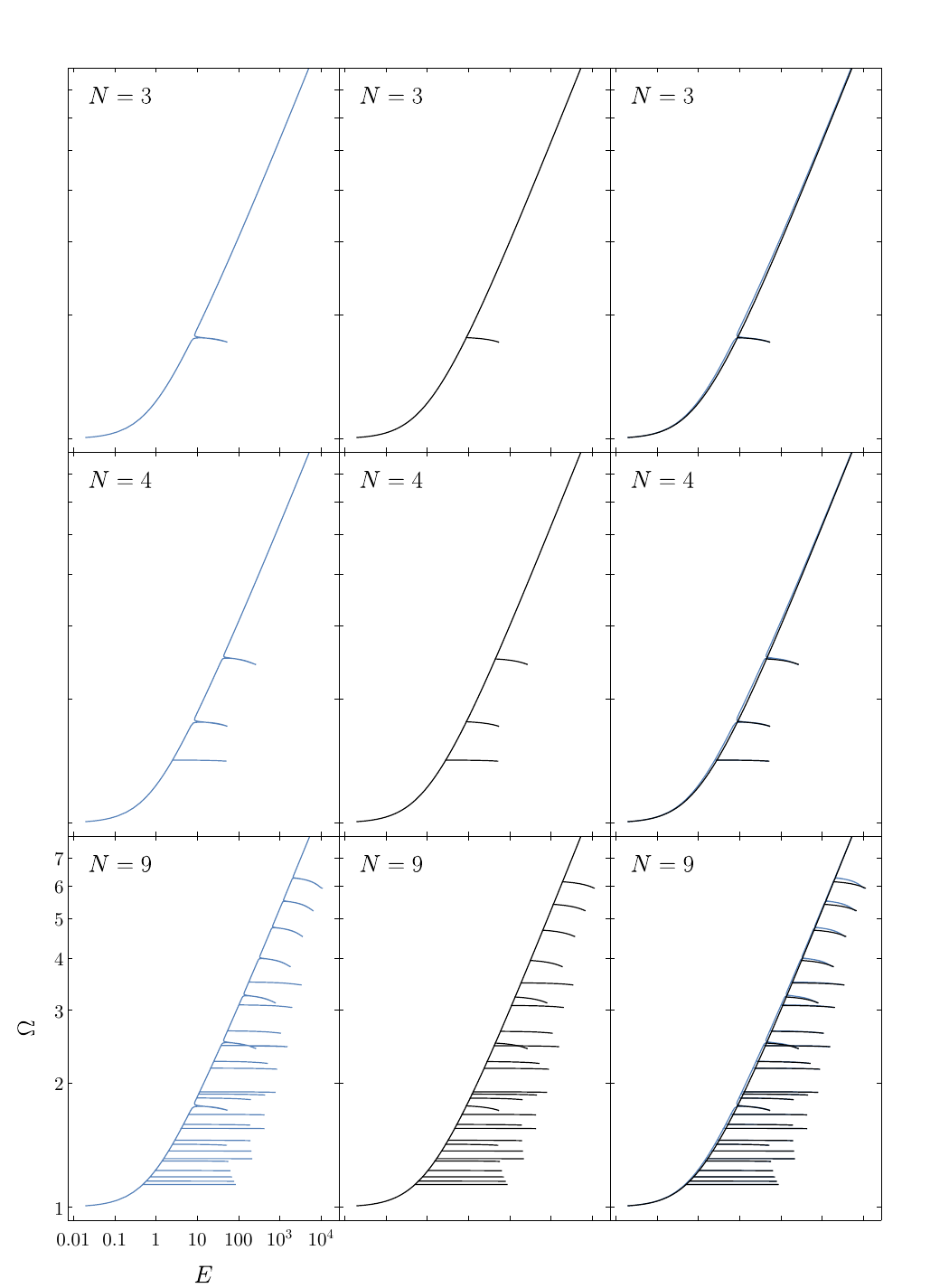}
	\caption{Comparison of solutions to the truncated Galerkin system (left column) and the corresponding $\Nt$-reducible tree (middle column) for increasing truncations $\Nt=3$ (top row), $\Nt=4$ (middle row), and $\Nt=9$ (bottom row). $\Nt$-reducible trees reproduce qualitatively the structure of solutions to the truncated Galerkin system. This is illustrated in the right column, where plots from the left and middle columns are superimposed for comparison, refer also to Fig.~\ref{fig:EOmegaZoomN4N9} for further details. Moreover, these plots show how complexity of the reducible tree increases with truncation.}
	\label{fig:EOmegaN349}
\end{figure}

First, we inspect the reducible tree constructed using its key building blocks made of the lowest dimensional solutions of reducible systems in Sec.~\ref{sec:ReducibleSystems}. For convenience we introduce the notion of a \textit{shoot} which is a point on a primary branch with the highest energy. Thus a shoot is a bifurcation point laying on a corresponding secondary trunk, see the red point in the left panel of Fig.~\ref{fig:2modePlot}.

When the size of the reducible tree grows new structures get attached to it, see middle column in Fig.~\ref{fig:EOmegaN349}. The complexity of the $\Nt$-reducible tree quickly increases with $\Nt$: for three-reducible tree we observe a single branch, for four-reducible tree three branches and 30 branches for nine-reducible tree (including 28 primary branches, one secondary branch, and a single branch of order three). 
The last case is the smallest tree for which we observe a three-modes solution previously presented in Fig.~\ref{fig:3modePlot}, see also Fig.~\ref{fig:EOmegaZoomN4N9}.

Based on the analysis of the previous sections we make the following observations for the $\Nt$-reducible tree when $\Nt$ increases:
\begin{itemize}
	\item the number of primary branches grows as $(\Nt-2)(\Nt-1)/2$,
	\item the primary branches of type $(m,n)$, see \eqref{eq:24.06.18_01}, bifurcate from the primary trunk \eqref{eq:24.05.11_05} at frequency $\Omega$ given by $\sqrt{(12n^2+12n-1)/(12m^2+12m-1)}$, with $n>m$, cf. the upper limit in \eqref{eqn:2modes_Omega}, and they end up at the shoot with the energy
	\begin{equation}
		\label{eq:24.07.25_01}
		E= \frac{16 \pi  (2 m+1)^2 (16 n^2+16n+1) (n-m) (m+n+1)}{3 (16 m^2+16m+1)^2}\,,
	\end{equation}
	see Eq.~\eqref{eq:24.06.25_01}\,,
	\item the point where the lowest laying branch bifurcates from the primary trunk approaches $\Omega=1$ as $1/\Nt$, while for the highest branch it moves up to infinity as $\mathcal{O}\left(\Nt\right)$,
	\item the energy at the shoot of the lowest laying primary branch grows as $\mathcal{O}\left(\Nt\right)$, cf. \eqref{eq:24.07.25_01} with $(m,n)=(\Nt-2,\Nt-1)$.
\end{itemize}

Using this information we conclude the following for the infinite ($\Nt=\infty$) reducible tree:
\begin{itemize}
	\item it contains infinitely many branches including also branches of higher order,
	\item bifurcation points of primary branches densely populate the primary trunk,
	\item there exist primary branches arbitrarily close to $\Omega=1$ such that energy on them becomes arbitrarily large as we approach the corresponding shoot.
\end{itemize}

The PDE symmetry \eqref{eq:24.05.13_01} transforms reducible systems into other reducible systems and so produces a rescaled reducible tree with $(m,n)$-trunk playing the role of the primary trunk. Primary branches of the reducible tree connect to the trunk of rescaled trees at their shoots, and so produce a ``fractal-like'' structure. Additionally, this same symmetry can be used to generate reducible trees for solutions bifurcating from higher linearised frequencies $\Omega=n$, with $n=2,3,\ldots\,$. This increases the complexity of the structure of periodic solutions even further.

Next, we compare a solution of the truncated Galerkin system with the corresponding $\Nt$-reducible tree. In Fig.~\ref{fig:EOmegaN349} we present the cases $\Nt=3,4$, and $9$. One notices that not only every primary branch is replicated by the reducible tree but also higher order structures are reproduced, as seen in Fig.~\ref{fig:EOmegaZoomN4N9}. Although, this approach provides an excellent qualitative description it does not capture all the details from the truncated Galerkin system. This is expected, as the reducible approximation is based on sparse mode interactions. The solution of the truncated Galerkin system has a form of a smooth curve tracing the reducible tree. In particular, the branches resolve smoothly into the loops, see Fig.~\ref{fig:EOmegaZoomN4N9}, which is an effect of perturbations by other modes (this issue is discussed in the next section). However, perturbations do not affect significantly shoots, as they still play the role of bifurcation points connecting the rescaled solutions.
The predictions of the reducible analysis improve as we consider solutions with frequencies closer to $\Omega=1$.

These observations back up Conjecture 1. Together with properties of the infinite reducible tree they let us make the statements formulated in the introduction about solutions to the PDE.

\begin{figure}[h!]
	\centering
	\includegraphics[width=0.49\textwidth]{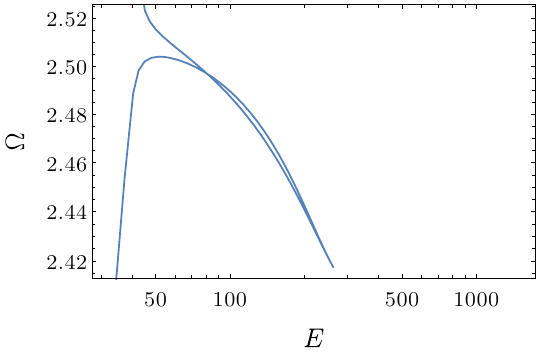}
	\includegraphics[width=0.49\textwidth]{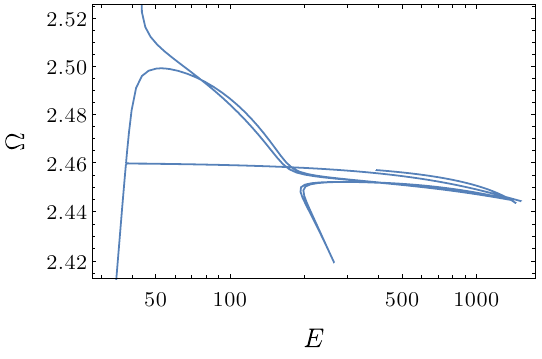}

	\includegraphics[width=0.49\textwidth]{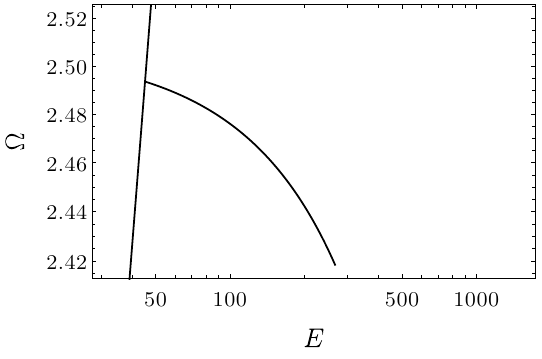}
	\includegraphics[width=0.49\textwidth]{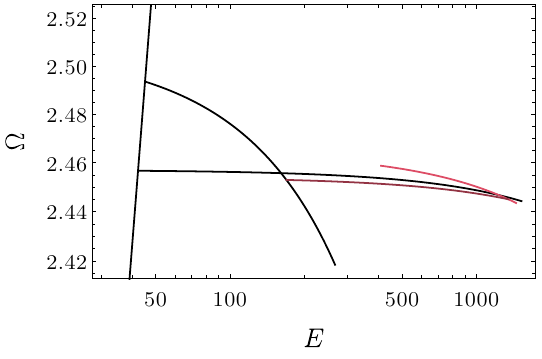}

	\includegraphics[width=0.49\textwidth]{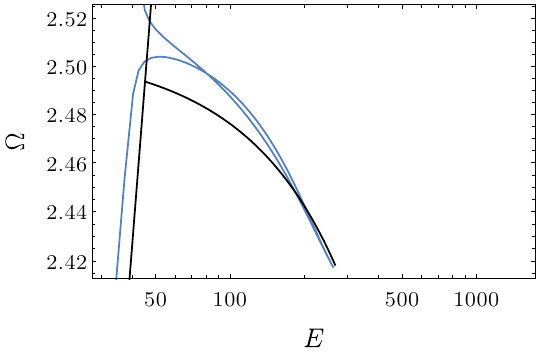}
	\includegraphics[width=0.49\textwidth]{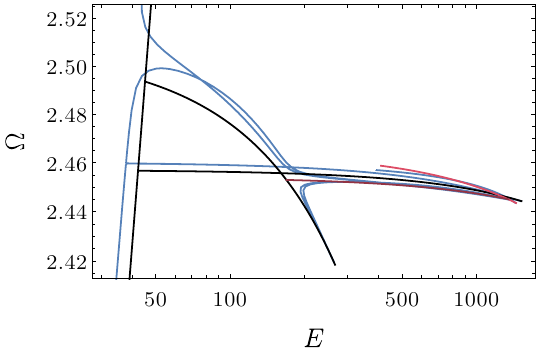}

	\caption{Zoom on one of the branches shown in Fig.~\ref{fig:EOmegaN349}, for truncations $\Nt=4$ (left) and $\Nt=9$ (right).
	In blue we plot solutions of the truncated Galerkin system, while other colours indicate the elements of the $\Nt$-reducible tree.
	When increasing $\Nt$, new primary branches appear. Besides, new features, described by branches of higher order (in red) as well as branches not connected to the primary trunk (in dark red) emerge. The $\Nt=9$ case is the lowest truncation for which a branch of order $3$ appears as a part of the reducible tree.
	}
	\label{fig:EOmegaZoomN4N9}
\end{figure}

\section{Perturbations of reducible systems}
\label{sec:PerturbationsOfReducibleSystems}

The goal of this section is to understand qualitatively how the presence of additional modes influences the structure of the reducible tree. In particular, we want to see how the branches get transformed into the loops observed in the previous section.

As the base case, we investigate the finite Galerkin system spanned by modes $A\cos \tau\, \sin x$, $B\cos 3 \tau\, \sin 5 x$, and $C \cos \tau \, \sin 3x$. Here the first and second modes set up a reducible system, while the third one plays the role of a perturbation. As a consequence, we get the following system of algebraic equations
\begin{subequations}{\label{eqn:pert_system}}
  \begin{empheq}[left=\empheqlbrace]{align}
A \left[9 A^2-9A C+12B^2-6BC+18C^2-16\Omega^2+16\right]+3 B C^2&=0\,, \label{eqn:pert_system_a}\\
B\left[9B^2+12A^2+12C^2-144\Omega^2+400\right]-3A^2C+3AC^2&=0\,, \label{eqn:pert_system_b}\\
C\left[9C^2+18A^2+12B^2+6AB-16\Omega^2+144\right]-3A^3-3A^2 B&=0\,.\label{eqn:pert_system_c}
  \end{empheq}
\end{subequations}
Let us point out, that the introduction of the perturbation breaks one of the symmetries of the reducible system. For the reducible system, if $(A,B)$ is a two-modes solution, then so are $(A,-B)$, $(-A,B)$, and $(-A,-B)$. System \eqref{eqn:pert_system} lacks this property, instead it possesses just the overall $\mathbb{Z}_2$ symmetry. If we treat $C$ as a perturbation, this loss of the symmetry lets us expect that the primary branch becomes two distinct solution curves in $E-\Omega$ plot. These solutions have been presented in Fig.~\ref{fig:Perturbation_2sys}. Indeed, the tree structure laid out by the reducible systems unties due to the presence of the perturbation. In particular, it means that some of the branches resolve into loops in $(E,\Omega)$ space.

\begin{figure}[h]
    \centering
    \includegraphics[width=0.95\textwidth]{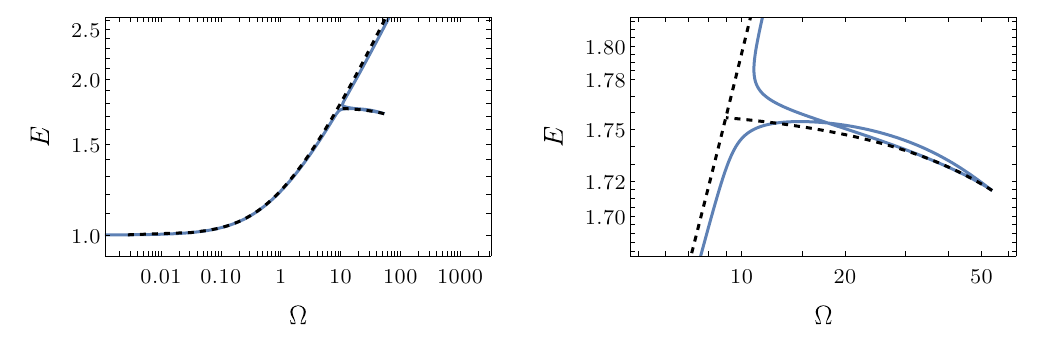}
    \includegraphics[width=0.95\textwidth]{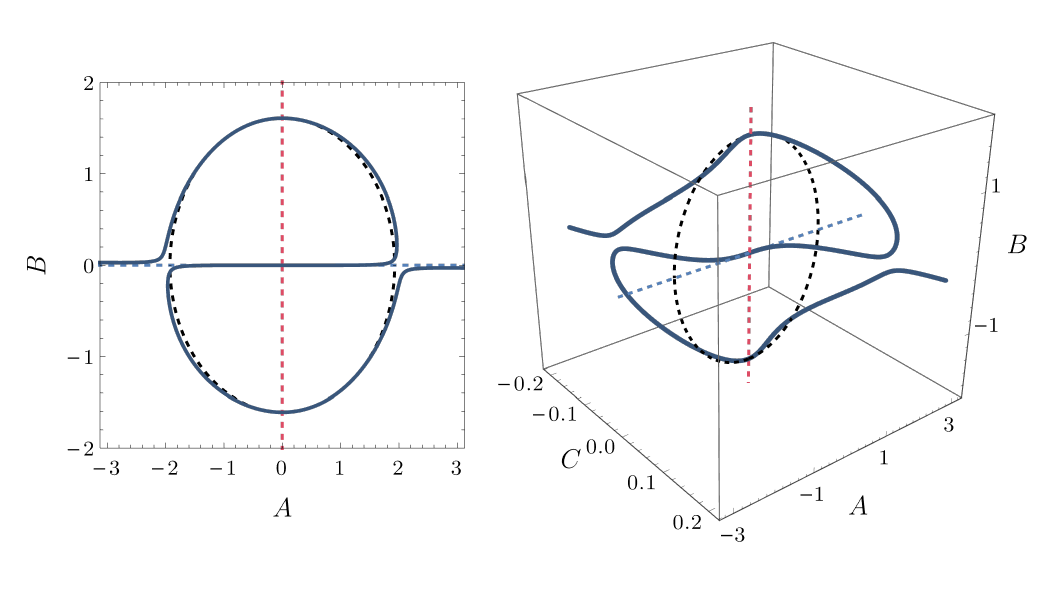}
    \caption{Solutions of perturbed reducible system \eqref{eqn:pert_system}. The solution curves of the corresponding reducible system on the $E-\Omega$ diagram (dashed black) resolve into a loop (solid blue). The bottom plots illustrate how the perturbation affects the structure coming from the reducible systems (dashed lines), see Fig.~\ref{fig:2modePlot}.}
    \label{fig:Perturbation_2sys}
\end{figure}

An analogous result can be observed by perturbing the considered system with the $C \cos 3\tau \, \sin 3x$ mode instead of $C \cos \tau \, \sin 3x$, but not with other modes. To understand which perturbations resolves the branch into the loop, let us investigate system \eqref{eqn:pert_system} in a greater detail. It still has one-mode solutions with $A=C=0$ and
\begin{equation*}
B=\pm\frac{4}{3}\sqrt{9\Omega^2-25}\, .
\end{equation*}
However, the primary trunk, previously one-mode solution given by \eqref{eq:24.05.11_05}, changes drastically. It still bifurcates from the same point but now in addition to $A\neq 0$ it must also hold $C\neq 0$, due to the $-3A^3$ term in \eqref{eqn:pert_system_c}, and $B\neq0$, as implied by terms with no $B$ in \eqref{eqn:pert_system_b}, along it. Since the primary trunk contains now both $A$ and $B$ modes, it is possible for it to smoothly follow the primary branch of type $\{(0,0),(1,2)\}$ of the reducible system and to meet with the secondary trunk, as can be seen in Fig.\ \ref{fig:Perturbation_2sys}. The main role played by the perturbation $C$ here is the introduction of additional mixing between $A$ and $B$ in \eqref{eqn:pert_system} via the presence of the terms mentioned above. A simple analysis shows us that, as already mentioned, the same effect takes place when $\cos \tau\, \sin 3 x$ is replaced by $\cos 3 \tau\,  \sin 3 x$, since then we get the system
\begin{equation*}
\left\{\begin{aligned}
A \left[9 A^2-3A C+12B^2+12C^2-12BC-16\Omega^2+16\right]&=0\,,\\
B\left[9B^2+12A^2+18C^2-144\Omega^2+400\right]-6A^2C&=0\,,\\
C\left[9C^2+12A^2+18B^2-144\Omega^2+144\right]-A^3-6A^2 B&=0\,.
\end{aligned}\right.
\end{equation*}
For any other perturbation one does not get this structure, for example, replacing $\cos \tau\,  \sin 3 x$ with $\cos 3 \tau\,  \sin x$ leads to
\begin{equation*}
\left\{\begin{aligned}
A \left[9 A^2+9A C+12B^2+18C^2-16\Omega^2+16\right]&=0\,,\\
B\left[9B^2+12A^2+18C^2-144\Omega^2+400\right]&=0\,,\\
C\left[9C^2+18A^2+18B^2-144\Omega^2+16\right]+3A^3&=0\,,
\end{aligned}\right.
\end{equation*}
and, as one can check, in this case the primary branch does not resolve into the loop.

In the construction above, the perturbative mode was able to give the system with the required structure because it was chosen in a particular way: it interacted with both of the original modes via the nonlinearity in a nontrivial way, see discussion under Definition~1. When the two modes of the reducible system are distant enough, for example when we consider a system spanned by modes $A \cos\tau\,\sin x$ and $B\cos 5\tau \,\sin 9x$, it is impossible to achieve this effect with an addition of a single mode. Then, one needs to include the whole ladder of perturbative modes in order to get the resolution of branches into loops. To see it, we may continue with this example and consider the system spanned by modes $A \cos\tau\,\sin x$, $B\cos 5\tau \,\sin 9x$, $C_1 \cos \tau \, \sin 3x$, and  $C_2 \cos 3 \tau \, \sin 5x$:
\begin{equation*}
\left\{\begin{aligned}
A \left[9A^2-9A C_1+12B^2+18C_1^2+12C_2^2-6C_1 C_2-16\Omega^2+16\right] \qquad\qquad\qquad &
\\-6B C_1 C_2 +3C_1^2 C_2 +3B C_2^2&=0\,,\\
B\left[9B^2+12A^2+12C_1^2+12 C_2^2-400\Omega^2+1296\right]+3AC_2^2-6A C_1 C_2&=0\,,\\
C_1\left[9C_1^2+18A^2+12B^2+12C_2^2+6A C_2-16\Omega^2+144\right]-3A^2 C_2-6ABC_2-3A^3&=0\,,\\
C_2\left[9C_2^2+12A^2+12B^2+12C_2^1+6A B-144\Omega^2+400\right]-3A^2 C_1+3A C_1^2-6ABC_1&=0\,.
\end{aligned}\right.
\end{equation*}
Here we can easily see that the primary branch in addition to $A\neq 0$ has to have $B\neq 0$, as in the case of the system \eqref{eqn:pert_system}. Indeed, $A\neq 0$ leads to $C_1$ and $C_2$ being nonzero, via two last equations. Then the second equation of the system gives us $B\neq 0$. Analogous perturbative ladders exist also for other reducible systems.

\appendix

\section{Analysis of two-modes systems}
\label{sec:TwoModeSystemAnalysis}

Section \ref{sec:2modes} focuses on reducible two-modes systems with one of the modes being $\cos\tau\, \sin x$ leading in particular to primary branches of order 2. In such a case the reducibility condition requires from the mode $\cos (2m+1) \tau\, \sin (2n+1) x$ to have $m, n\geq 1$ while $m>1$ or $n>1$. For completeness, here we consider other two-modes systems, i.e. those which do not fall into the reducible category. Thus, we consider finite Galerkin systems spanned by modes $A\cos\tau \sin x$ and $B\cos (2m+1) \tau\, \sin (2n+1)x$ with $m,n\geq 0$. We partially adapt the nomenclature from the main text, in particular, the solution bifurcating from zero at $\Omega=1$ with $A\neq 0$ will be called the primary trunk (independently of whether $B=0$ or not).
\begin{theorem}
    The primary trunk has an additional bifurcation point at $\Omega>1$ if and only if $m, n\geq 1$ and $n>m$.
\end{theorem}

\begin{proof}
The case when $m, n\geq 1$ and $m>1$ or $n>1$ has been covered in the main text. Here we analyse the remaining possible choices of $m$ and $n$.

\begin{customcase}{1}[System spanned by modes $A\cos\tau \sin x$ and $B\cos\tau\, \sin 3x$]
Here the Galerkin system becomes
\begin{subequations}{\label{eqn:appCase1}}
  \begin{empheq}[left=\empheqlbrace]{align}
A \left[9 A^2-9AB+18 B^2-16\Omega^2+16\right]&=0\,, \label{eqn:appCase1a}\\
B \left[18 A^2+9 B^2-16\Omega^2+144\right]-3A^3&=0\,. \label{eqn:appCase1b}
  \end{empheq}
\end{subequations}
Putting $A=0$ solves \eqref{eqn:appCase1a} and leads to either $B=0$ or a one-mode solution given by
\begin{equation}
\label{eq:24.06.12_02}
B=\pm \frac{4}{3} \sqrt{\Omega^2-9}\,.
\end{equation}
This solution bifurcates from zero at $\Omega=3$. Alternatively, \eqref{eqn:appCase1a} can be solved by
\begin{equation*}
A=\frac{B}{2}\pm\frac{1}{6}\sqrt{64(\Omega^2-1)-63 B^2}\,.
\end{equation*}
The following calculations can be carried on independently of the choice of the sign in this expression. For definiteness, we choose the plus sign. Then, plugging such $A$ into \eqref{eqn:appCase1b} leads to a third-order polynomial in $B^2$:
\begin{align}
	\label{eq:24.06.12_01}
    35721 B^6-54432 \left(\Omega ^2+1\right)B^4+20736\left(\Omega ^4+2 \Omega ^2+29\right)B^2-2048 \left(\Omega ^2-1\right)^3 = 0
    \,,
\end{align}
It is immediate to see that one of its solutions is a primary trunk. The discriminant of \eqref{eq:24.06.12_01} is given by
\begin{align*}
    \Delta=20643001562824704 \left(\Omega ^2+23\right)^2 \left(\Omega ^8-4 \Omega ^6-186 \Omega ^4-132 \Omega ^2-2751\right)
    \,.
\end{align*}
From the Descartes' Rule of Signs it follows that $\Delta$ has one positive root $\Omega_{*}$. One can check that it is located between $\Omega=4$ and $\Omega=\sqrt{17}$. For $\Omega<\Omega_{*}$ there exists one value of $B^{2}$ giving a solution to \eqref{eq:24.06.12_01}, the primary trunk of \eqref{eqn:appCase1}, while for $\Omega>\Omega_{*}$ there are three solutions.
In the following we show that the two extra solutions for $\Omega>\Omega_{*}$, neither the secondary trunk \eqref{eq:24.06.12_02} do not cross the primary trunk.

Let us consider the hyperboloid in $(\Omega,A,B)$ space given by
\begin{equation*}
9(A^2+B^2)=16\left((\Omega+1)^2-9\right).
\end{equation*}
As we will see, no solution crosses it, except for the trivial one $\Omega=2$, $A=B=0$. In particular, it means that the primary trunk lies outside it, while the secondary one is inside. Then we will show that additional roots bifurcate from the secondary trunk, concluding the proof. We consider the following system of equations
\begin{subequations}
  \begin{empheq}[left=\empheqlbrace]{align}
9 A^2-9AB+18 B^2-16\Omega^2+16&=0\,, \label{eqn:appCase1Ha}\\
B \left[18 A^2+9 B^2-16\Omega^2+144\right]-3A^3&=0\,, \label{eqn:appCase1Hb}\\
9(A^2+B^2)-16\left((\Omega+1)^2-9\right)&=0\,. \label{eqn:appCase1Hc}
  \end{empheq}
\end{subequations}

One can calculate $B$ from \eqref{eqn:appCase1Hc} and plug it into \eqref{eqn:appCase1Ha} and \eqref{eqn:appCase1Hb}. Again, we have the freedom of the choice of the sign, but without the loss of generality we may choose positive $B$. As a result, \eqref{eqn:appCase1Ha} becomes
\begin{equation*}
3A \sqrt{16 \left(\Omega ^2+2 \Omega -8\right)-9 A^2}=9A^2-16\left(\Omega^2+4\Omega-15\right)\,.
\end{equation*}
This square root is present also in \eqref{eqn:appCase1Hb}, so we can get rid of it there. Then, after squaring both sides of the equation above, we get the following system
\begin{subequations}
  \begin{empheq}[left=\empheqlbrace]{align}
-18 A^4+16 A^2 \left(3 \Omega ^2+10 \Omega -38\right)-\frac{256}{9} \left(\Omega ^2+4 \Omega -15\right)^2&=0\,, \label{eqn:appCase1H2a}\\
-27 A^4+36 A^2 \left(\Omega ^2+2 \Omega -16\right)+64 \left(2 \Omega ^3+9 \Omega ^2-26 \Omega -15\right)&=0\,. \label{eqn:appCase1H2b}
  \end{empheq}
\end{subequations}
Then we get
\begin{equation*}
A^2 \left[A^2 (81-6 \Omega  (\Omega +7))+8 \left(\Omega ^4+12 \Omega ^3-160 \Omega +202\right)\right]=0\,,
\end{equation*}
as a linear combination of \eqref{eqn:appCase1H2a} and \eqref{eqn:appCase1H2b}. By disregarding $A=0$ (for this case we have the analytic formula telling us that the solution does not cross the hyperboloid), we get a simple expression for $A^2$ that can be plugged into  \eqref{eqn:appCase1H2b}, finally giving us a single polynomial equation in $\Omega$
\begin{equation*}
(2\Omega+1)(\Omega ^6+6 \Omega ^5-84 \Omega ^4+40 \Omega ^3+2328 \Omega ^2-6588 \Omega +4902)=0\,.
\end{equation*}
Now we want to find all positive roots of the sixth order polynomial there. By evaluating it at points $\Omega=-12$, $\Omega=-10$, $\Omega=0$, one can check that it has at least two negative roots. The discriminant of this polynomial is positive, hence all its roots are simple. By constructing the Sturm sequence we can show that the mentioned two negative roots are the only two real roots of this polynomial. Hence, it does not have any positive roots and, as a result, no non-trivial solutions of \eqref{eqn:appCase1} cross the hyperboloid.

The last step is to show that the additional solutions emerging for $\Omega>\Omega_{*}$ are inside the hyperboloid. We do it by showing that those solutions bifurcate from the secondary trunk. Let us introduce $x=A$, $y=B-4\sqrt{\Omega^2-9}/3$. Then \eqref{eqn:appCase1} can be rewritten as $F(\Omega,x,y)=0$, where $F:\mathbb{R}\times\mathbb{R}^2\to\mathbb{R}^2$ is
\begin{align*}
F(\Omega,x,y)=&\left(x(9x^2+16(\Omega^2-17)-9xy+18y^2+12\sqrt{\Omega^2-9}(4xy-x^2)),\right.\\
&\left.-3x^2+32y(\Omega^2-9)+9y^2+18x^2 y+12 \sqrt{\Omega^2-9}(2x^2+3y^2)\right)\,.
\end{align*}
Then for $\Omega\geq 3$, $F(\Omega,0,0)=0$ is just the secondary trunk. This form is ready to use the standard local bifurcation theory to show that $\Omega=\sqrt{17}$ indeed is the bifurcation point, from which two additional branches bifurcate.
\end{customcase}

\begin{customcase}{2}[System spanned by modes $A\cos\tau \sin x$ and $B\cos\tau\,\sin (2n+1)x$ with $n\geq 2$]
This case is simple, since the system of equations becomes
\begin{equation*}
\left\{\begin{aligned}
A \left[9 A^2+18 B^2-16\Omega^2+16\right]&=0\,, \\
B \left[18 A^2+9 B^2-16\Omega^2+16(2n+1)^2\right]&=0\,.
\end{aligned}\right.
\end{equation*}
Apart from the trivial zero solution, there also exist two one-mode solutions: the primary trunk
\begin{equation*}
A=\pm \frac{4}{3}\sqrt{\Omega^2-1}\,, \qquad B=0\,,
\end{equation*}
and  the secondary trunk emerging from $\Omega=2n+1$
\begin{equation*}
A=0\,, \qquad B=\pm\frac{4}{3}\sqrt{\Omega^2-(2n+1)^2}\,.
\end{equation*}
Analysis analogous to the one presented in the main text shows that for $\Omega>\sqrt{8 n^2+8 n+1}$ there exists an additional solution bifurcating from the secondary trunk and given by
\begin{equation*}
A=\pm \frac{4}{3\sqrt{3}}\sqrt{\Omega^2-(8n^2+8n+1)}\,, \qquad B=\pm\frac{4}{3\sqrt{3}}\sqrt{\Omega^2-(4n^2+4n-1)}\,.
\end{equation*}
This solution never crosses the primary trunk.
\end{customcase}

\begin{customcase}{3}[System spanned by modes $A\cos\tau \sin x$ and $B\cos 3\tau\, \sin x$]
Here we get
\begin{subequations}
  \begin{empheq}[left=\empheqlbrace]{align}
A \left[9 A^2+9AB+18 B^2-16\Omega^2+16\right]&=0\,, \label{eqn:appCase3a}\\
B \left[18 A^2+9 B^2-144\Omega^2+16\right]+3A^3&=0\,. \label{eqn:appCase3b}
  \end{empheq}
\end{subequations}
When $A=0$, we either have $B=0$ or
\begin{equation*}
B=\pm 4\sqrt{\Omega^2-\frac{1}{9}}\,,
\end{equation*}
i.e., a one-mode solution bifurcating from zero at $\Omega=1/3$. The other possible solution of \eqref{eqn:appCase3a} can be written as
\begin{equation*}
A=-\frac{B}{2}\pm\frac{1}{6}\sqrt{64(\Omega^2-1)-63 B^2}\,.
\end{equation*}
Plugging this expression (without the loss of generality we can choose the plus sign) into \eqref{eqn:appCase3b} leads to a third-order polynomial in $B^2$
\begin{align*}
	35721 B^6+54432 \left(\Omega ^2+1\right)B^4+20736  \left(1+2 \Omega ^2+29 \Omega ^4 \right)^2 B^2-2048 \left(\Omega ^2-1\right)^3\,.
\end{align*}
For $\Omega<1$ all its terms have positive coefficients, hence, it has no positive roots. Using the Descartes’ Rule of Signs we conclude that for $\Omega>1$ this polynomial has one positive solution which gives us a primary trunk bifurcating from $\Omega=1$.
\end{customcase}

\begin{customcase}{4}[System spanned by modes $A\cos\tau \sin x$ and $B\cos 3\tau\, \sin 3x$]
This time the finite system has the form
\begin{subequations}
  \begin{empheq}[left=\empheqlbrace]{align}
A \left[9 A^2-3AB+12 B^2-16\Omega^2+ 16\right]&=0\,, \label{eqn:appCase4a}\\
B \left[12 A^2+9 B^2 -144\Omega^2+ 144\right]-A^3&=0\,. \label{eqn:appCase4b}
  \end{empheq}
\end{subequations}
We can solve \eqref{eqn:appCase4a} by putting $A=0$, then $B$ must be either equal to zero or
\begin{equation*}
B=\pm4\sqrt{\Omega^2-1}\,.
\end{equation*}
Additional solutions to \eqref{eqn:appCase4a} can be written as
\begin{equation*}
A=\frac{B}{6}\pm\frac{1}{6}\sqrt{64(\Omega^2-1)-47 B^2}\,.
\end{equation*}
Plugging one of these expressions into \eqref{eqn:appCase4b} leads to a third-order polynomial equation in $B^2$:
\begin{align}\label{eqn:appCase4P}
	24759 B^6+485568 \left(\Omega ^2-1\right)B^4+5577984  \left(\Omega ^2-1\right)^2 B^2-2048 \left(\Omega ^2-1\right)^3 = 0
	\,
\end{align}
with discriminant
\begin{align*}
	\Delta=-9853595053368871187644416(\Omega^2-1)^6
	\,.
\end{align*}
For $\Omega=1$ the only root of \eqref{eqn:appCase4P} is $B=0$. If $\Omega\neq 1$ we have $\Delta<0$ so \eqref{eqn:appCase4P} has a single real solution. Since the value of \eqref{eqn:appCase4P} at $B=0$ is $-2048 \left(\Omega ^2-1\right)^3$, it is negative for $\Omega<1$ and positive for $\Omega>1$. As a result, for $\Omega>1$ there is a solution with both modes $A$ and $B$ that bifurcates from $\Omega=1$ (primary trunk). Together with the trunk $A=0$ also emerging from $\Omega=1$ they are all real solutions of the system.
\end{customcase}

\begin{customcase}{5}[System spanned by modes $A\cos\tau \sin x$ and $B\cos (2m+1)\tau\, \sin x$ with $m\geq 2$]
This case is similar to Case 2, because the system has the form
\begin{equation*}
\left\{\begin{aligned}
A \left[9 A^2+18 B^2 -16\Omega^2 + 16\right]&=0\, ,\\
B \left[18 A^2+9 B^2 -(2m+1)^2\Omega^2 + 16\right]&=0\, .
\end{aligned}\right.
\end{equation*}
In addition to the trivial zero solution, it has also two one-mode solutions. One of them bifurcates from $\Omega=1$ and is given by
\begin{align*}
    A=\pm \frac{4}{3}\sqrt{\Omega^2-1}\,, \qquad B=0\,,
\end{align*}
while the other one bifurcates from zero at $\Omega=1/(2m+1)$
\begin{align*}
    A=0\,, \qquad B=\pm \frac{4}{3}\sqrt{(2m+1)^2\Omega^2-1}\,.
\end{align*}
The two-modes solution given by
\begin{align*}
    A=\pm \frac{4}{3\sqrt{3}} \sqrt{(8m^2+8m+1)\Omega^2-1}\,, \qquad B=\pm \frac{4}{3\sqrt{3}} \sqrt{(1-4m-4m^2)\Omega^2-1}\,.
\end{align*}
is not real for $m\geq 2$, hence it has to be excluded. Here we have only a primary and secondary trunk.
\end{customcase}

\end{proof}

\printbibliography

\end{document}